\documentclass[a4paper,11pt]{article}
\usepackage{amsfonts}
\usepackage{amsmath}
\usepackage{amsthm}
\usepackage{natbib}
\usepackage{graphicx}
\usepackage[margin=1.5in]{geometry}

\providecommand{\keywords}[1]{\textbf{Keywords:} #1}
\renewcommand\footnotemark{}
\title{Robustness properties of marginal composite likelihood estimators}
\author{Helen Ogden \thanks{This work was supported by the
Engineering and Physical Sciences Research Council [grant numbers EP/P50578X/1, EP/K014463/1]}}
\date{University of Warwick, Coventry, UK}

\newtheorem{theorem}{Theorem}
\newtheorem{lemma}{Lemma}

\theoremstyle{definition}
\newtheorem{example}{Example}

\begin{document}

\maketitle

\begin{abstract}
Composite likelihoods are a class of alternatives to the full likelihood
which are widely used in many situations in which the likelihood itself is intractable.
A composite likelihood may be computed without the need to specify the full
 distribution of the response, which means that in some situations the
resulting estimator will be more robust to model misspecification than the maximum likelihood estimator.
The purpose of this note is to show that such increased robustness is not guaranteed. An example
is given in which various marginal composite likelihood estimators are inconsistent under
model misspecification, even though the maximum likelihood estimator is consistent.
\end{abstract}

\keywords{Consistency; Generalized linear mixed model; Random-effects misspecification}

\section{Introduction}

Suppose we observe independent samples $y^{(1)}, \ldots, y^{(n)}$,
where each $y^{(i)}=(y^{(i)}_1,\ldots,y^{(i)}_m)$ is assumed to be a sample from
 a model depending on
 an unknown parameter $\theta$. The likelihood
\[L(\theta) = \prod_{i=1}^n L(\theta;y^{(i)})\]
is sometimes difficult to compute, and composite likelihoods \citep{Lindsay1988}
provide a class of alternatives for conducting inference about $\theta$ in such circumstances.
A marginal composite likelihood
\[L^C(\theta) = \prod_{i=1}^n \prod_{s \in S} L(\theta;y^{(i)}_s)^{w_s}\]
is formed by taking a product of component likelihoods, each
of which is the likelihood given some subset of the data $y_s$,
where $w_s$ is a weight assigned to component $s$. We will assume from now on that $w_s=1$ for
all $s$, although similar results will hold for other choices of the weights.
There are a large number of possible choices for the subsets $s$ contained in $S$.
One option is to take each $s$ to be a pair of elements $\{j,k\}$, giving a pairwise likelihood \citep{Cox2004}.
A review of composite likelihoods and their many uses
is given by \cite{Varin2011}.

If the model is correctly specified, the composite
likelihood estimator will be consistent as $n \rightarrow \infty$, provided that the parameter
of interest remains identifiable, although the estimator will typically have a higher asymptotic variance than
the maximum likelihood estimator.
There is a hope that some compensation for this loss of efficiency may be provided
by an increased robustness of the composite likelihood estimator to misspecification of the model \citep{Varin2011,Xu2011}.
This is motivated by the fact that it is not necessary to specify the full distribution of the
response in order to be able to compute a composite likelihood.
If the marginal distributions of
$Y_{s}$ for each subset $s \in S$ are correctly specified,
 then the corresponding estimator
of $\theta$ will be consistent as $n \rightarrow \infty$, even if the full model
 is misspecified. The maximum likelihood estimator
 need not be consistent in such a setting, since the likelihood relies on the full, misspecified, distribution
of $Y$.

In some situations the relevant marginal distributions themselves may be misspecified, in which case
the marginal composite likelihood estimator no longer retains this robustness property.
However, in that case the full distribution of $Y$ must also be incorrect,
so the maximum likelihood estimator need not be robust to this misspecification either.
 From this, it is tempting to conclude that a marginal composite likelihood
estimator must always be at least as robust to model misspecification as the full likelihood estimator.
We give an example to demonstrate that this is not the case. Our example is a generalized linear mixed model,
under misspecification of the random-effects distribution. We consider asymptotics as the number of random
effects and the amount of information on each random effect simultaneously tend to infinity. The maximum likelihood estimator of the regression coefficient is consistent in this setting, even if the random-effect distribution is misspecified, but
various marginal composite likelihood estimators are not consistent under the same misspecification.

\section{Robustness under random-effect misspecification}
\subsection{A two-level model}
We consider a generalized linear mixed model, with
two-level nested structure.
Suppose that there are $m$ observations $y^{(i)}=(y^{(i)}_{1},\ldots,y^{(i)}_{m})$
on each of $n$ items $i=1,\ldots,n$, and that covariates $x_i$ are recorded for each item.
 The distribution of the response depends on the covariates through a linear predictor
$\eta_i$, and conditional on knowledge of $\eta$,
the $m$ observations on each item are independent. We model
$\eta_i = \beta^T x_i + b_i,$
where $b_i$ are independent random effects, which are assumed to be drawn
from a $N(0,\sigma^2)$ distribution.
We are interested in the robustness of estimators of $\beta$ to deviations from this assumed
random-effects distribution.
 We suppose that each $b_i$ is actually drawn independently from a distribution such that
$E(b_i) = 0$ and $\text{var}(b_i) < \infty$.

Under this misspecification, we are interested in the limit
 $\beta^\infty_m$ of the maximum likelihood estimator of $\beta$ as $n \rightarrow \infty$,
  and how this limit varies with $m$.
In particular, we will show that as $m \rightarrow \infty$, $\beta^\infty_m \rightarrow \beta_0$,
so that if $n$ and $m$ simultaneously tend to infinity, the maximum likelihood estimator
is consistent.

\subsection{Consistency of the maximum likelihood estimator}

The likelihood for $\theta=(\beta,\sigma)$ is given by
$L(\theta;y) = \prod_{i=1}^n L_i(\theta;y^{(i)}),$
where
\[L_i(\theta;y^{(i)}) = \int_{-\infty}^\infty \left\{ \prod_{j=1}^{m} f_y(y^{(i)}_{j} \mid \eta_i = \beta^T x_i + b_i) \right\} \frac{1}{\sigma} \phi\left(\frac{b_i}{\sigma}\right) db_i,\]
where $\phi(.)$ is the standard normal density function.
We write $\ell_i(\theta;y^{(i)}) = \log L(\theta;y^{(i)})$ for the contribution to the log-likelihood from $y^{(i)}$, and $u_i(\theta;y) = \nabla_\theta \ell_i(\theta;y)$ for the corresponding score function.  In the case that $m$ is fixed and $n \rightarrow \infty$,
the results of \cite{White1982} show that the maximum likelihood estimator of $\theta$ will converge to the
value $\theta^{\infty}_m$ which solves $\bar u (\theta) = E\left\{u_i(\theta,Y^{(i)})\right\} = 0$, where the expectation is taken over the true distribution of $Y^{(i)}$ and the covariates $x_i$.

Intuitively, for large $m$, it should possible to obtain an estimate of
 the value of each linear predictor $\eta_i$  from the data $y^{(i)}$, which will be close to
 the true value $\eta_i^0$. This means that for sufficiently large $m$,
 inference given the data $y$
 should be similar to the inference we would obtain if we were given the true value $\eta_i^0$
 of each linear predictor $\eta_i$, and assume a linear model
 $\eta_i = \beta^T x_i + b_i$, where $b_i$ are assumed to be independent $N(0,\sigma^2)$ error terms. Thus,
 for large $m$, the problem is reduced to studying the impact of incorrectly assuming that the
 errors in a linear model are normally distributed.

To formalize this argument,
write
\[\ell_i(\theta;\eta_i^0) = -\frac{1}{2} \log (2 \pi \sigma) - \frac{1}{2 \sigma^2}(\eta_i^0-\beta^T x_i)^2\]
for the log-likelihood for $\theta$ in the linear model given $\eta_i^0$, and
$u_i(\theta;\eta^{0}_i) = \nabla_\theta \ell_i(\theta;\eta_i^0)$ for the
 corresponding score function. We obtain the following result, whose proof is given in the appendix.
\begin{lemma}
\label{thm:same_as_lm}
As $m \rightarrow \infty$,
\[u_i(\theta;y^{(i)}) = u_i(\theta;\eta_i^0) +o(1).\]
\end{lemma}

In the linear model setting, the impact of the distribution of the error term has been well studied.
We may use the results from this setting to show that the asymptotic bias in $\hat \beta$ diminishes
with $m$.
\begin{lemma}
As $m \rightarrow \infty$, $\beta^\infty_m \rightarrow \beta_0$.
\label{thm:bias_tends_to_zero}
\end{lemma}
\begin{proof}
If we observe $\eta_i$ from the linear model $\eta_i = \beta^T x_i + b_i$,
the ordinary least squares estimator
of $\beta$, which is identical to the maximum likelihood estimator under the assumption of
 normally distributed errors $b_i$,
  is consistent irrespective of the true distribution of the error $b_i$.
 This means that
$\bar u(\theta^*;\eta) = 0$ is solved by $\theta^* = (\beta_0,\sigma^*)$, for some $\sigma^*$.
But
 \[\bar u(\theta^*;y) = \bar u(\theta^*;\eta)+o(1) = o(1) \rightarrow 0\]
 as $m \rightarrow \infty$. So $\beta^\infty_m \rightarrow \beta_0$ as $m \rightarrow \infty,$
 as claimed.
\end{proof}

As a consequence of Lemma \ref{thm:bias_tends_to_zero}, we have the following consistency result
\begin{theorem}
As $n, m \rightarrow \infty$, $\hat \beta_m^n \rightarrow \beta_0$ in probability.
\end{theorem}
The maximum likelihood estimator is consistent in this setting because the increasing
amount of information on each random effect allows us to obtain a good estimate
of each $\eta_i$.
A composite likelihood estimator which uses only small subsets of the data in each component likelihood
 does not make use of this increasing amount of information on each random effect, and so
will not be consistent under the same misspecification.

\subsection{Inconsistency of fixed-order composite likelihood estimators}
Recall that we construct a marginal composite likelihood by taking a product
of the densities of some subsets of the observed data. It is possible to construct
an auxiliary model for a new data-vector $\tilde y$, so that the likelihood for $\theta$ given
$\tilde y$ in the auxiliary model is identical to the composite likelihood for $\theta$ given $y$
in the original model.
To do this, we write $\tilde y = \{y_{s}, s \in S\}$, and model the components
indexed by different $s$ as independent, with $Y_{s}$ having the same marginal distribution
as in the original model. This construction allow us to find the limit of the composite likelihood
estimator in the original model by using the limit of the maximum likelihood estimator
in the auxiliary model. For example, the limit of the pairwise likelihood estimator of $\beta$
 as $n$ and $m$ tend to infinity
is $\beta^\infty_2$, the limit of the maximum likelihood estimator when $m$ is fixed at two.
In most cases, the pairwise likelihood estimator will be inconsistent as $n$ and $m$ tend to infinity, unless
the random-effect distribution is correctly specified.

\begin{example}
Consider making $m$ repeated observations on each of $n$ items $i=1,\ldots,n$, each of which
has a binary covariate $x_i$ associated with it, where $x_i \sim \text{Bernoulli}(1/2)$. Suppose
that each observation is binary, and that $pr(Y^{(i)}_{j}=1 \mid \alpha,\beta,\sigma,u_i,x_i) = \Phi(\alpha + \beta x_i+ b_i),$
where we suppose $b_i \sim N(0,\sigma^2)$.
Instead of treating
 $\sigma$ as an unknown parameter of interest, we fix it at some constant value $\tilde \sigma$. We suppose that in truth,
 $b_i \sim N(0,\sigma_0^2)$, where $\sigma_0 \not = \tilde \sigma$, and consider the impact of this misspecification. While such misspecification may be unrealistic in practice, it nonetheless provides a concrete example of the above asymptotic
 results.

Figure \ref{fig:alpha_lim} shows a contour plot of the limit $\alpha^\infty_m$ for various values of $m$ and $\tilde \sigma$,
when $\alpha_0=0.5$, $\beta_0 = 1$ and $\sigma_0 = 0.5$.
As expected from the theoretical results, for each fixed $\tilde \sigma \not = \sigma_0$, the asymptotic bias in $\hat \alpha$
diminishes with $m$.
Figure \ref{fig:alpha_pair_lim} shows a cut across this contour plot at $m=2$, which gives the limit of the pairwise
likelihood estimator of $\alpha$. If $\tilde \sigma \not = \sigma_0$, the pairwise likelihood estimator of $\alpha$ is
 not consistent. The corresponding results for $\beta$ are similar. The limiting value of a similarly defined $k$-wise likelihood estimator may be obtained
by a cut across the contour plot at $m=k$.  In the limit as $n$ and $m$ tend to infinity, the maximum likelihood estimator is consistent, but the
$k$-wise likelihood estimator is not, for any fixed $k$.

\begin{figure}
\includegraphics[width=\textwidth]{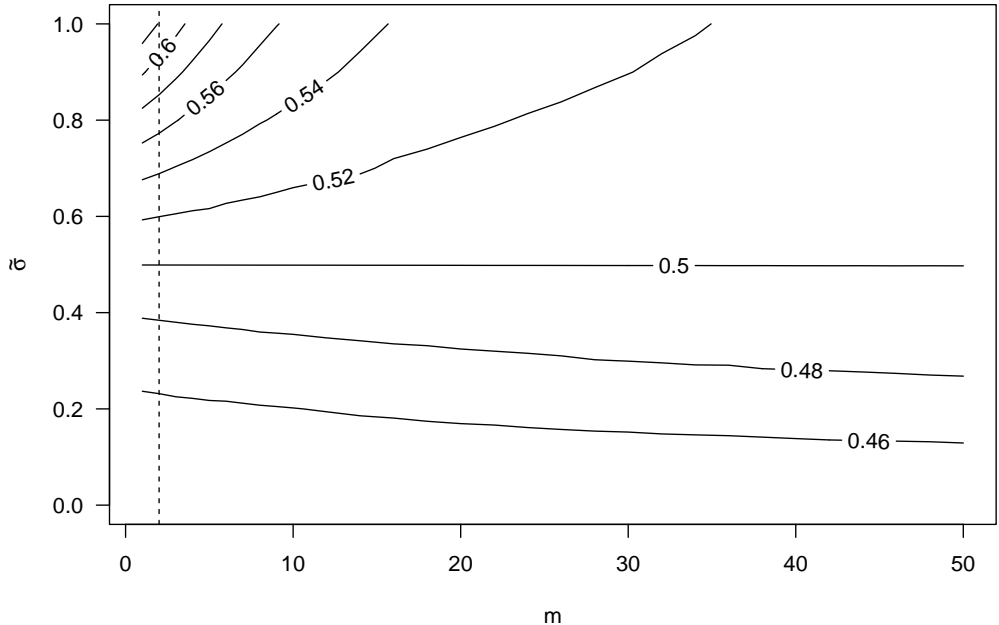}
\caption{The limit of $\hat \alpha$ as $n\rightarrow \infty$, for various values of $m$ and $\tilde \sigma$.
The dashed line at $m=2$ shows how to find the limit of the pairwise
likelihood estimator.}
\label{fig:alpha_lim}
\end{figure}

\begin{figure}
\includegraphics[width=\textwidth]{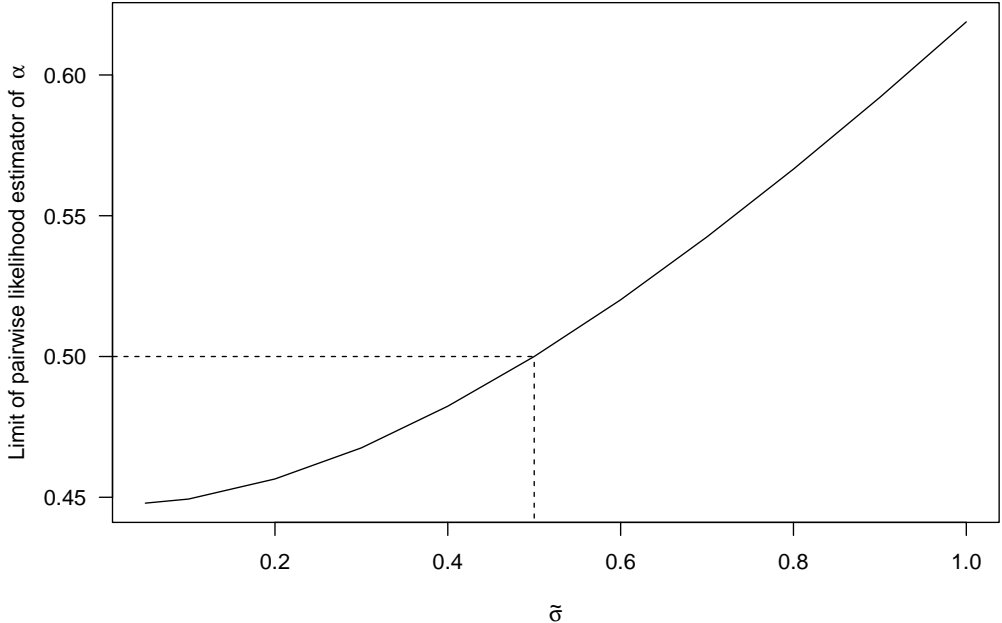}
\caption{The limit of the pairwise likelihood estimator of $\alpha$, for various values of $\tilde \sigma$.
The estimator is only consistent when $\tilde \sigma = \sigma_0$.}
\label{fig:alpha_pair_lim}
\end{figure}
\end{example}
\section{Discussion}
The example given is intended
to provide a warning against the notion that a marginal composite likelihood estimator will always
be at least as robust to model misspecification as the maximum likelihood estimator.
In the two-level model described, it is relatively straightforward to obtain the full likelihood, so
there is little computational motivation for using a composite likelihood in this case.
However, in models with more complex structure,
the high-dimensional integral in the likelihood no longer factorizes into a product of one-dimensional
integrals. Composite likelihoods have been proposed to provide a computationally feasible alternative to full likelihood inference for such models \citep{Bellio2005}, and similar results
to those obtained here for the two-level model can be expected to hold in these settings.

The results also provide a contribution to the literature on the impact of random-effect misspecification, by
making explicit the impact of the sparsity of the model on the robustness
of the maximum likelihood estimator to misspecification of the random-effects distribution.



\appendix
\label{appendix:proof}
\section*{Appendix}
\subsection*{Proof of Lemma \ref{thm:same_as_lm}}
\begin{proof}
We first reparameterize the integral, to write
the likelihood as an integral over $\eta_i$.
Let
\[g_i(\eta_i \mid y^{(i)},\theta) = \left[ \prod_{j=1}^{m} f_y(y^{(i)}_{j} \mid \eta_i) \right] \frac{1}{\sigma}
\phi \left(\frac{\eta_i - \beta^T x_{i}}{\sigma} \right),\]
so that
$L_i(\theta ; y^{(i)}) = \int_{-\infty}^\infty g_i(\eta_i \mid y^{(i)},\theta) d\eta_i.$
We may think of $g_i(. \mid y^{(i)},\theta)$ as a non-normalised posterior density for $\eta_i$, given
a prior \[\frac{1}{\sigma} \phi\left(\frac{\eta_i - \beta^T x_i}{\sigma}\right)\] which shrinks $\eta_i$ towards $\beta^T x_i$.
Provided that $\sigma>0$, as $m$ increases, the $\prod_{j=1}^{m} f_y(y^{(i)}_{j} \mid \eta_i)$ term, which does
not depend on $\theta$, dominates the prior, so that $\hat \eta_i(\theta)$,
the maximizer of $g_i(. \mid y^{(i)},\theta)$ over $\eta_i$, loses its dependence on $\theta$, and tends
towards its true value $\eta_i^0 = \beta_0^T x_{i} + b_i^0$.

As $m$ increases, $g_i(. \mid y^{(i)},\theta)$ becomes well approximated by a normal density
about $\hat \eta_i(\theta)$, and the relative error in a Laplace approximation to $L_i(\theta;y^{(i)})$ tends to zero.
Writing $\ell_i(\theta ; y^{(i)})=\log L_i(\theta ; y^{(i)})$,
as $m \rightarrow \infty$,
\[ \ell_i(\theta ; y^{(i)})= \log g_i(\hat\eta_i(\theta) \mid y^{(i)},\theta)
+\frac{1}{2} \log H_\theta(\hat\eta_i(\theta)) - \frac{1}{2}\log 2 \pi + o(1),\]
where
\[H_\theta (\eta_i) = \frac{\partial^2}{\partial \eta_i^2} \log g_i(\eta_i \mid y^{(i)},\theta).\]
So, for any two distinct $\theta_1$, $\theta_2$, the difference in log-likelihoods $\ell_i(\theta_1 ; y^{(i)})-\ell_i(\theta_2 ; y^{(i)})$ is
\begin{align*}
\log &\, g_i(\hat\eta_i(\theta_1) \mid y^{(i)},\theta_1) - \log g_i(\hat\eta_i(\theta_2) | y^{(i)},\theta_2)
+  \frac{1}{2} \log H_{\theta_1}(\hat\eta_i(\theta_1))-\frac{1}{2} \log H_{\theta_2}(\hat\eta_i(\theta_2)) +o(1) \\
& = \log g_i(\eta_i^0 \mid  y^{(i)},\theta_1) - \log g_i(\eta_i^0 \mid  y^{(i)},\theta_2)
+  \frac{1}{2} \log H_{\theta_1}(\eta_i^0)-\frac{1}{2} \log H_{\theta_2}(\eta_i^0) +o(1),
\end{align*}
since for any $\theta$, $\hat\eta_i(\theta) \rightarrow \eta_i^0$ in probability as $m \rightarrow \infty$. But
\[\frac{H_{\theta_2}(\eta_i)}{H_{\theta_1}(\eta_i)} = \frac{\frac{1}{m}  \left[ \sum_{j=1}^{m} \frac{\partial^2}{\partial \eta_i^2} \log f_y(y_{ij}\mid \eta_i)
+ \frac{\partial^2}{\partial \eta_i^2} \log \left\{ \frac{1}{\sigma_2}
\phi \left(\frac{\eta_i - \beta_2^T x_{i}}{\sigma_2}\right) \right\} \right]}
{\frac{1}{m}  \left[ \sum_{j=1}^{m} \frac{\partial^2}{\partial \eta_i^2} \log f_y(y_{ij}\mid \eta_i) + \frac{\partial^2}{\partial \eta_i^2} \log \left\{ \frac{1}{\sigma_1}
\phi \left(\frac{\eta_i - \beta_1^T x_{i}}{\sigma_1}\right) \right\} \right]} \rightarrow 1\]
in probability as $m \rightarrow \infty$, so
\begin{align*}
\ell_i(\theta_1 ;  y^{(i)}) - \ell_i(\theta_2 ; y^{(i)}) &=
 \log g_i(\eta_i^0\mid y^{(i)},\theta_1) - \log g_i(\eta_i^0\mid y^{(i)},\theta_2) +o(1) \\
 &= \log \left\{ \frac{1}{\sigma_1} \phi \left(\frac{\eta_i^0 - \beta_1^T x_i}{\sigma_1} \right) \right\}
 -\log \left\{ \frac{1}{\sigma_2} \phi \left(\frac{\eta_i^0 - \beta_2^T x_i}{\sigma_2} \right) \right\} +o(1) \\
 &= \ell_i(\theta_1 ;  \eta_{i}^0) - \ell_i(\theta_2;  \eta_i^0) +o(1).
\end{align*}
Letting $\theta_1 = \theta$, $\theta_2 = \theta+h$ and considering the limit as $h \rightarrow 0$, we therefore have
$u_i(\theta ;  y^{(i)}) = u_i(\theta ; \eta_{i}) +o(1),$ as claimed.
\end{proof}
\bibliographystyle{apalike}
\bibliography{references}

\end{document}